\documentclass{amsart}
\usepackage{hyperref}
\hypersetup{bookmarksnumbered} 
\usepackage{breakurl}
\usepackage{amssymb,amsmath,amsthm}
\usepackage{graphicx}
\usepackage[subrefformat=parens,labelformat=parens]{subfig}

\input xy
\xyoption{all}

\usepackage{color}

\newcommand{\blue}[1]{\textcolor{blue}{#1}}

\makeatletter

\let\c@table\c@figure
\makeatother

\newtheorem{theorem}{Theorem}

\newtheorem{proposition}[theorem]{Proposition}
\newtheorem{lemma}[theorem]{Lemma}
\newtheorem{corollary}[theorem]{Corollary}

\newtheorem*{theproblem}{Problem 14.10(a)}

\theoremstyle{definition}

\theoremstyle{remark}
\newtheorem{remark}[theorem]{Remark}

\newcommand{\Tbar}{\mkern 1.5mu\overline{\mkern-1.5mu T}}

\newcommand{\Z}{\mathbb{Z}}
\newcommand{\N}{\mathbb{N}}
\newcommand{\R}{\mathbb{R}}
\newcommand{\Q}{\mathbb{Q}}

\newcommand{\Aut}{\mathrm{Aut}}
\newcommand{\TA}{T\!\hspace{0.1em}\mathcal{A}}
\newcommand{\VA}{V\!\hspace{0.1em}\mathcal{A}}

\newcommand{\PL}{\mathrm{PL}}

\title{Embedding $\mathbb{Q}$ into a Finitely Presented Group}

\author{James Belk}
\address{Department of Mathematics, Cornell University, Ithaca, New York 14853.}
\email{\href{mailto:jmb226@cornell.edu}{jmb226@cornell.edu}}

\thanks{The first author has been partially supported by EPSRC grant EP/R032866/1 as well as the National Science Foundation under Grant No.~DMS-1854367 during the creation of this paper.}

\author{James Hyde}
\address{Department of Mathematics, Cornell University, Ithaca, New York 14853.}
\email{\href{mailto:jth263@cornell.edu}{jth263@cornell.edu}}

\author{Francesco Matucci}
\address{Dipartimento di Matematica e Applicazioni, Universit\`{a} degli Studi di Milano--Bicocca, Milan 20125, Italy.}
\email{\href{mailto:francesco.matucci@unimib.it}{francesco.matucci@unimib.it}}
\thanks{The third author is a member of the Gruppo Nazionale per le Strutture Algebriche, Geometriche e le loro Applicazioni (GNSAGA) of the Istituto Nazionale di Alta Matematica (INdAM) and gratefully acknowledges the support of the 
Funda\c{c}\~ao para a Ci\^encia e a Tecnologia  (CEMAT-Ci\^encias FCT projects UIDB/04621/2020 and UIDP/04621/2020) and of the Universit\`a degli Studi di Milano--Bicocca
(FA project ATE-2016-0045 ``Strutture Algebriche'').
}

\date{}  

\begin{document}

\maketitle

\begin{abstract}We observe that the group of all lifts of elements of Thompson's group~$T$ to the real line is finitely presented and contains the additive group~$\Q$ of the rational numbers.  This gives an explicit realization of the Higman embedding theorem for~$\Q$, answering a Kourovka notebook question of Martin Bridson and Pierre de~la Harpe.
\end{abstract}

\section*{Introduction}

In 1961, Graham Higman proved that any countable group with a computable presentation can be embedded into a finitely presented group~\cite{Hig}.
For example, the additive group $\Q$ of rational numbers has computable presentation
\[
\langle s_1,s_2,s_3,\ldots \mid s_n^n=s_{n-1}\text{ for all }n\geq 2\rangle
\]
and can therefore be embedded into some finitely presented group.  Unfortunately, Higman's construction is difficult to carry out in practice, and group presentations produced by his procedure are  quite large and unwieldy.

Higman was for many years interested in finding more explicit embeddings of various naturally occurring recursively presented groups such as $\Q$ into finitely presented groups~\cite{Johnson}. 
In 1999, the following question was submitted to the Kourovka notebook~\cite{Kourovka} and labeled as a ``well-known problem''.  The question is attributed to Pierre de la Harpe in~\cite{Kourovka}, but Martin Bridson and de la Harpe have informed us that the question was originally submitted jointly by the two of them.

\begin{theproblem}It is known that any recursively presented group embeds in a finitely presented group.  Find an explicit and ``natural'' finitely presented group\/ $\Gamma$ and an embedding of the additive group of the rationals\/ $\Q$ in\/~$\Gamma$.
\end{theproblem}

The problem then asks the same question for the group~$\mathrm{GL}_n(\Q)$.  The problem originally included a part~(b) that asked for any finitely \textit{generated} example, although such examples had already been supplied by Hall in 1959~\cite{Hall}.  In particular, 
Hall observed that if $W$ is a vector space over $\Q$ with basis $\{e_n\}_{n\in\Z}$ and $\sigma,\rho$ are the linear transformations of $W$ defined by $\sigma(e_n)=e_{n+1}$ and 
$\rho(v_n) = p_ne_n$, where $\{p_n\}_{n\in\Z}$ is some enumeration of the primes, then the orbit of $e_0$ under $\langle\sigma,\rho\rangle$ generates~$W$ as an abelian group, and hence
the semidirect product $W\rtimes\langle\sigma,\rho\rangle$ is finitely generated and contains~$\Q$.
Further finitely generated examples were later supplied by Mikaelian~\cite{Mik}. 
As for embeddings into finitely presented groups, Mikaelian~\cite{Mik2} has described how to explicitly carry out Higman's construction for $\mathbb{Q}$ as well as many other groups of interest, such as  $\mathrm{GL}_n(\mathbb{Q})$.

In this note we observe that $\Q$ embeds into a finitely presented group $\Tbar$ which was introduced by Ghys and Sergiescu in 1987~\cite{GhSe}.  This is an explicit group of homeomorphisms of the real line, which has a presentation with two generators and four relators (see Remark~\ref{rem:FinitePresentation} below). Specifically, $\Tbar$~consists of all homeomorphisms $f\colon \R\to \R$ that satisfy the following conditions:
\begin{enumerate}
    \item The homeomorphism $f$ is piecewise-linear, with finitely many breakpoints on each compact interval.\smallskip
    \item Each linear portion of $f$ has the form $f(x)=2^n x+d$, where $n\in\Z$ and $d$ is a dyadic rational.\smallskip
    \item Each breakpoint of $f$ has dyadic rational coordinates.\smallskip
    \item The homeomorphism $f$ commutes with the translation $x\mapsto x+1$.  That is, $f(x+1)=f(x)+1$ for all $x\in\R$.
\end{enumerate}
It follows from (4) that the set of breakpoints of $f$ is invariant under $x\mapsto x+1$, and hence any $f\in \Tbar$ is either linear or has infinitely many breakpoints. As a group of orientation-preserving homeomorphisms of the real line, $\Tbar$~is torsion-free, and indeed right-orderable~\cite[Theorem~6.8]{GhysSurvey}.  The monomorphisms $\Q\to \Tbar$ that we describe below are order-preserving. 

The group $\Tbar$ is closely related to the three groups $F$, $T$, and $V$ introduced by Richard J.~Thompson in the \mbox{1960's \cite{Tho, CFP}}. Thompson's group~$F$ arises naturally as the ``group of associative laws'' and also arose independently in homotopy theory~\cite{FreydHeller}, while Thompson's groups $T$ and $V$ were the first known examples of infinite, finitely presented simple groups.  Thompson's group~$T$ is the group of homeomorphisms of the circle $\R/\Z$ that satisfy conditions (1), (2), and (3) above, and~$\Tbar$ is precisely the group of all ``lifts'' of elements of $T$ to the real line.  In particular, the quotient of $\Tbar$ by the cyclic subgroup generated by $x\mapsto x+1$ is isomorphic to~$T$.  This cyclic subgroup is precisely the center of~$\Tbar$, and therefore $\Tbar$ is a central extension of~$T$ (though it is not the universal central extension).  Ghys and Sergiescu introduced $\Tbar$ in this context as part of their investigation into the cohomology of~$T$~\cite{GhSe}. 

Thompson made the surprising observation that $T$ contains elements of arbitrary finite order~\cite{Tho}.  In~2011, Bleak, Kassabov, and the third author gave an elementary argument that  $\Q/\Z$ embeds into Thompson's group~$T$~\cite{BlKaMa}. They did not consider the consequences for the group~$\Tbar$, but it follows easily from their result that $\Q$ embeds into~$\Tbar$.  We give a self-contained proof of this below, and indeed we prove something a bit stronger:

\begin{theorem}
The group $\Tbar$ has continuum many different subgroups isomorphic to~$\Q$, all of which contain the center of~$\Tbar$.
\end{theorem}

Brin has proven that~$\Tbar$ embeds naturally into the automorphism group of Thompson's group~$F$.  Specifically, Brin proved \cite[Theorem~1]{BrinChameleon} that $\mathrm{Aut}(F)$ has an index-two subgroup $\mathrm{Aut}_+(F)$ which is isomorphic to the group of all homeomorphisms of~$\mathbb{R}$ that satisfy conditions (1), (2), and (3) above and agree with elements of $\Tbar$ in neighborhoods of $-\infty$ and~$\infty$.  Brin also showed \cite[Theorem~1]{BrinChameleon} that this group $\mathrm{Aut}_+(F)$ fits into a short exact sequence
\[
1\to F \to \mathrm{Aut}_+(F) \to T\times T \to 1,
\]
and it follows easily that $\mathrm{Aut}(F)$ is finitely presented. Indeed, Burillo and Cleary have computed an explicit finite presentation for $\Aut(F)$ in~\cite{BuClAuto}.
This gives another natural example of a finitely presented group that contains~$\Q$:

\begin{corollary}The automorphism group of Thompson's group~$F$ has a subgroup isomorphic to\/~$\Q$.
\end{corollary}

\begin{remark}\label{rem:FinitePresentation}
The smallest known presentation for Thompson's group $T$ has two generators and five relators, and was derived by Lochak and Schneps in~\cite{LoSc}. (Note that the version in~\cite{LoSc} contains a typo.  See \cite[Proposition~1.3]{FunKap} for a corrected version).  Using this presentation together with the fact that $\Tbar$ is a central extension of $T$ by~$\mathbb{Z}$, it is not difficult to derive a presentation for $\Tbar$ with two generators and four relators.  Specifically, 
\[
\Tbar = \bigl\langle a,b \;\bigl|\; a^4=b^3,(ba)^5=b^9, [bab,a^2baba^2]=[bab,a^2b^2a^2baba^2ba^2]=1\bigr\rangle
\]
where $a$ and $b$ are the elements of $\Tbar$ whose restrictions to $[0,1]$ are defined by
\[
a(x) = \begin{cases}
\frac{1}{2}x+\frac{1}{2} & \text{if }0\leq x \leq \frac{3}{4}, \\[3pt]
x+\frac{1}{8} & \text{if } \frac{3}{4} < x \leq \frac{7}{8}, \\[3pt]
4x-\frac{5}{2} & \text{if }\frac{7}{8}<x\leq 1,
\end{cases}
\qquad
b(x) = \begin{cases}
\frac{1}{2}x+\frac{1}{2} & \text{if } 0 \leq x \leq \frac{1}{2}, \\[3pt]
x+\frac{1}{4} & \text{if }\frac{1}{2}< x \leq \frac{3}{4}, \\[3pt]
2x-\frac{1}{2} & \text{if } \frac{3}{4} < x \leq 1.
\end{cases}
\]
\end{remark}

\begin{remark}In addition to the above results, we have obtained an explicit embedding of $\Tbar$ and hence an embedding of~$\Q$ into a finitely presented simple group~$\TA$, verifying the Boone-Higman conjecture in the case of~$\Q$.  We have also obtained an explicit finitely presented simple group $\VA$ that contains all countable abelian groups. Both of these groups will be described in a forthcoming paper. 
\end{remark}

\begin{remark}
Hurley~\cite{Hur} and Ould Houcine~\cite{Hou} have proven that there exists a finitely presented group $G$ whose center is isomorphic to~$\mathbb{Q}$.  It would be interesting to find a natural example of such a group, or at least a natural example of a finitely presented group whose center contains~$\mathbb{Q}$.
\end{remark}

\subsection*{Acknowledgments} The authors would like to thank Collin Bleak for many helpful
conversations and suggestions about this work.  We would also like to thank Matthew Brin and Matthew Zaremsky for their comments on an early draft of this manuscript and Martin Bridson and Pierre de la Harpe for comments
on the historical perspective.  Finally, we would like to thank an anonymous referee for many helpful comments and suggestions.

\section*{Inclusion of \texorpdfstring{$\Q$}{Q} into \texorpdfstring{$\Tbar$}{Tbar}}\label{sec:Inclusion}

Let $\PL_2(\R)$ denote the (uncountable) group of all piecewise-linear homeomorphisms of $\R$ that satisfy conditions (1) through (3) for elements of $\Tbar$ given in the introduction.  The group $\Tbar$ is precisely the centralizer in $\PL_2(\R)$ of the homeomorphism $z(x)=x+1$.

If $[a,b]$ and $[c,d]$ are closed intervals in~$\R$, we say that a piecewise-linear homeomorphism $h\colon [a,b]\to [c,d]$ is \textit{Thompson-like} if it is a restriction of an element of~$\PL_2(\R)$, i.e.~if it satisfies conditions (1) through (3) for elements of $\Tbar$ given in the introduction.  It is well-known that if $[a,b]$ and $[c,d]$ have dyadic rational endpoints, then there exists at least one Thompson-like homeomorphism $[a,b]\to[c,d]$ (cf.~\cite[Lemma~4.2]{CFP}).

\begin{lemma}\label{lem:RootsThompsonLike}
Let $g$ be an element of\/ $\PL_2(\R)$ without fixed points and let $n\geq 2$. Then there exist infinitely many different $f\in\PL_2(\R)$ such that~$f^n=g$.
\end{lemma}
\begin{proof}
Without loss of generality, suppose that $g(0)>0$.  Choose dyadic rationals $0=p_0<p_1<\cdots <p_n=g(0)$, and for each $1\leq i< n$ choose a Thompson-like homeomorphism $f_i\colon [p_{i-1},p_i]\to [p_i,p_{i+1}]$. Let
$f_n\colon [p_{n-1},p_n]\to [p_n,g(p_1)]$ be the homeomorphism $gf_1^{-1}f_2^{-1}\cdots f_{n-1}^{-1}$, and let $f\in\PL_2(\R)$ be the homeomorphism that agrees with $f_i$ on each $[p_{i-1},p_i]$ ($1\leq i\leq n$) and satisfies
\[
f(x) = g^kfg^{-k}(x)
\]
for each $x\in[g^k(0),g^{k+1}(0)]$ with $k\ne 0$.

To prove that $f^n=g$, observe that on the interval $[p_{i-1},p_i]$, the function $f^n$ restricts to the composition
\[
(gf_{i-1}g^{-1})\cdots(gf_2g^{-1})(gf_1g^{-1})f_n\cdots f_{i+1}f_i
\]
Since $f_n=gf_1^{-1}f_2^{-1}\cdots f_{n-1}^{-1}$, the expression above simplifies to~$g$.
Thus $f^n$ agrees with $g$ on $[0,g(0)]$, and it follows easily that $f^n=g$. Moreover, since there are infinitely many possible choices for~$p_1,\ldots,p_{n-1}$ and $f_1,\ldots,f_{n-1}$, there
are infinitely many possibilities for~$f$.
\end{proof}

\begin{lemma}\label{lem:extract-root2}
Let $m\geq 1$ and let $g\in \Tbar$ so that $g^m=z$.  Then for every $n\geq 2$ there exist infinitely many different $f\in \Tbar$ so that $f^n=g$.
\end{lemma}
\begin{proof}
Note that $g$ cannot have any fixed points, since these would also be fixed points of~$z$. Therefore, by Lemma~\ref{lem:RootsThompsonLike}, there exist infinitely many $f\in\PL_2(\R)$ such that $f^n=g$.  Any such homeomorphism commutes with $z$ since $f^{mn}=z$, and therefore every such $f$ lies in~$\Tbar$.
\end{proof}

\begin{proposition}\label{prop:UncountablyMany}The group $\Tbar$ has continuum many subgroups isomorphic to~$\Q$.
\end{proposition}
\begin{proof}
Observe that $\Q$ has presentation
\[
\langle s_1,s_2,\ldots \mid s_n^n = s_{n-1}\text{ for }n\geq 2 \rangle.
\]
To obtain an embedding of $\Q$ into $\Tbar$ it suffices to find a sequence $\{s_n\}_{n\in\N}$ of elements of~$\Tbar$ such that $s_1$ has infinite order and $s_n^n = s_{n-1}$ for all $n\geq 2$.  Such a sequence can be defined recursively by letting $s_1=z$ and then repeatedly applying Lemma~\ref{lem:extract-root2} to find, for each $n\geq 2$, an element $s_n\in \Tbar$ such that $s_n^n=s_{n-1}$. Since there are infinitely many choices for~$s_n$ at each stage, this procedure constructs continuum many different copies of~$\Q$.
\end{proof}

\begin{remark}
Since each subgroup of $\Tbar$ is conjugate to only countably many other subgroups, it follows from Proposition~\ref{prop:UncountablyMany} that $\Tbar$ has continuum many conjugacy classes of subgroups isomorphic to~$\Q$.
\end{remark}

\begin{remark}\label{rem:SpecificExample}
The choice of the elements $s_n$ in the proof of Proposition~\ref{prop:UncountablyMany} can be carried out constructively.  For example, let $\{d_n\}_{n\in\N}$ be the decreasing sequence of dyadics in $[0,1]$ defined recursively by $d_1=1$ and $d_n=d_{n-1}/2^{n-1}$. Let $s_1=z$, and for each $n\geq 2$ let $s_n$ be the $n$th root of $s_{n-1}$ in $\PL_2(\R)$ that satisfies
\[
s_n(x) = \begin{cases}x+d_n & \text{if }0\leq x \leq d_n, \\[3pt]
2x & \text{if }d_n < x \leq \tfrac{1}{2}d_{n-1}.
\end{cases}
\]
Then the sequence $\{s_n\}_{n\in\N}$ generates a subgroup of $\Tbar$ isomorphic to~$\Q$.

It is possible to write these elements explicitly in terms of the generators for $\Tbar$ given in Remark~\ref{rem:FinitePresentation}.  Specifically, let $p$ be any element of $\Tbar$ which is the identity on $\bigl[0,\tfrac{1}{2}\bigr]$ and has slope $1/2$ on $\bigl[0,\tfrac{5}{8}\bigr]$ (e.g.~$p=a^{-1}b$), let $q$ be any element of $\Tbar$ which maps $\bigl[0,\tfrac{3}{8}\bigr]$ linearly to $\bigl[0,\tfrac{3}{4}\bigr]$ (e.g~$q=a^{-1}ba^2b^{-1}$), and let $r=b^{-1}aba^2(ab)^{-2}ba^{-1}b$ be the element of $\Tbar$ which satisfies
\[
r(x) = \begin{cases} x & \text{if }0\leq x \leq \tfrac{1}{4} \\[2pt]
\tfrac12 x +\tfrac{1}{8} & \text{if }\tfrac{1}{4}\leq x \leq \tfrac{1}{2} \\[2pt]
2x-\tfrac{5}{8} & \text{if }\tfrac{1}{2}\leq x\leq \tfrac{5}{8} \\[2pt]
x & \text{if }\tfrac{5}{8}\leq x\leq 1.
\end{cases}
\]
Define a sequence of elements $\{t_n\}_{n\geq 3}$ recursively by $t_3=b^2a(ab)^{-2}b$ and
\[
t_n = (t_{n-1} \triangleleft q^{n-2})\bigl(r\triangleleft p^{n-4}q^{n(n-3)/2}\bigr)
\]
for $n\geq 4$, where $x\triangleleft y$ denotes $y^{-1}xy$.  Then $t_n$ maps the left half of $[0,d_{n-1}]$ linearly to $[0,d_{n-1}]$, maps the right half of $[0,d_{n-1}]$ linearly to the left half of $[d_{n-1},d_{n-1}+d_n]$, maps $[d_{n-1},d_{n-1}+d_n]$ linearly to its right half, and is the identity on $[d_{n-1}+d_n,1]$.  The desired sequence $\{s_n\}$ can now be defined recursively by
\begin{align*}
s_n &= 
[t_n,t_n\triangleleft s_{n-1}] \bigl(t_n\triangleleft s_{n-1}^{1-(n-1)!}\bigr)\cdots \bigl(t_n\triangleleft s_{n-1}^{-2}\bigr)\bigl(t_n\triangleleft s_{n-1}^{-1}\bigr) t_n
\\[3pt] &= [t_n,t_n\triangleleft s_{n-1}]\,s_1\,\bigl(s_{n-1}^{-1}t_n\bigr)^{(n-1)!}
\end{align*}
for $n\geq 4$, where $[x,y]$ denotes $xyx^{-1}y^{-1}$, $s_3=b^{-1}aba^{-2}baba^{-1}b^{-1}$, $s_2=ba^2b^{-1}$, and $s_1=b^3$.

The idea here is that $t_n$ is roughly the same as $s_n$ on $[0,d_{n-1}]$ and is the identity elsewhere.  Since $[0,d_{n-1}]$ is a fundamental domain for the action of~$\langle s_{n-1}\rangle$,  we can construct $s_n$ by multiplying together conjugates of $t_n$ by powers of $s_{n-1}^{-1}$, with the correction factor $[t_n,t_n\triangleleft s_{n-1}]$ accounting for the overlap between the initial $t_n$ and the last conjugate $t_n\triangleleft s_{n-1}^{1-(n-1)!} = t_n \triangleleft s_{n-1}$.  The authors have checked all of the above computations in \textit{Mathematica}~\cite{code}.
\end{remark}

\begin{remark}
The copy of $\Q$ constructed in Remark~\ref{rem:SpecificExample} has the property that the orbit of $0$ is dense in $\R$.  For such a copy, the resulting action of $\Q$ on $\R$ is conjugate by a homeomorphism of $\R$ to the usual action of $\Q$ on $\R$ by translation. However, there are also ``exotic'' copies of $\Q$ in $\Tbar$ for which the orbit of $0$ is not dense in~$\R$. For example, we can choose a sequence $\{s_n\}_{n\in\N}$ in $\Tbar$ with $s_1=z$ and $s_n^n=s_{n-1}$ ($n\geq 2$) such that $s_n(0) = \tfrac{1}{2}+\tfrac{1}{2^n}$ for all~$n$.  In this case, the subgroup $\langle s_1,s_2,s_3,\ldots\rangle$ is isomorphic to~$\Q$, but the orbit of $0$ under the action of this subgroup does not intersect the interval $(0,1/2]$.  It follows that the restricted wreath product $F \wr_{\Q/\Z} \Q = \bigl(\bigoplus_{\Q/\Z} F\bigr)\rtimes \Q$ embeds into~$\Tbar$, where Thompson's group~$F$ embeds into $\Tbar$ as the group of elements that are the identity on~$[1/2,1]$.  Bleak, Kassabov, and the third author used a similar argument to prove that $F\wr (\Q/\Z)$ embeds into~$T$ \cite[Theorem~1.6]{BlKaMa}.
\end{remark}

\begin{remark}\label{rem:NoQInV}
Higman proved that elements of Thompson's group $V$ of infinite order do not have roots of arbitrarily large orders~\cite[Corollary~9.3]{Hig2}.  It follows that $\Q$ does not embed into~$V$, and hence $\Q$ does not embed into $T$, either.
\end{remark}

Every copy of $\Q$ obtained from the proof of Proposition~\ref{prop:UncountablyMany}
contains the center $\langle z\rangle$ of~$\Tbar$. The following proposition asserts that these are all of the subgroups of $\Tbar$ isomorphic to~$\Q$.

\begin{proposition}\label{prop:QSubgroupsContainCenter}
Every subgroup of $\Tbar$ isomorphic to\/~$\Q$ contains the center of $\Tbar$.
\end{proposition}

\begin{proof}
Let $A$ be a subgroup of $\Tbar$ isomorphic to~$\Q$.  Since $\Q$ does not embed into $T$ (see Remark~\ref{rem:NoQInV}), the projection homomorphism $\Tbar\to T$ cannot be injective on~$A$, so $A$ must intersect the center of~$\Tbar$ nontrivially.  In particular, $A$ must contain $z^n$ for some $n\geq 1$.  Since $A$ is isomorphic to~$\Q$, there exists an $f\in A$ so that~$f^n=z^n$.  Since $f$ and $z$ commute it follows that $(fz^{-1})^n =1$, and since $\Tbar$ is torsion-free we conclude that $fz^{-1}=1$, and therefore $A$ contains~$z$. 
\end{proof}

\bigskip
\newcommand{\arxiv}[1]{\href{https://arxiv.org/abs/#1}{\blue{arXiv:#1}}}
\newcommand{\doi}[1]{\href{https://doi.org/#1}{\blue{doi:#1}}}
\bibliographystyle{plain}

\end{document}